\documentclass[graybox]{svmult}

\usepackage{type1cm}        
\usepackage{makeidx}         
\usepackage{graphicx}        
\usepackage{multicol}        
\usepackage[bottom]{footmisc}

\usepackage{newtxtext}       %
\usepackage{newtxmath}       

\usepackage{enumerate,graphicx,xypic}
\makeindex             

\begin{document}

\title*{A Study of Directional Entropy Arising from \(\mathbb{Z} \times \mathbb{Z}_+\) Semigroup Actions}
\author{Hasan Ak\i n}
\institute{Hasan Ak\i n \at Department of Mathematics, Faculty of
Arts and Sciences, Harran University, Sanliurfa, TR63050, Turkey
\email{akinhasan@harran.edu.tr; akinhasan25@gmail.com}}
%
%
\maketitle  

\abstract*{Each chapter should be preceded by an abstract (no more
than 200 words) that summarizes the content. The abstract will
appear \textit{online} at \url{www.SpringerLink.com} and be
available with unrestricted access. This allows unregistered users
to read the abstract as a teaser for the complete chapter. Please
use the 'starred' version of the \texttt{abstract} command for
typesetting the text of the online abstracts (cf. source file of
this chapter template \texttt{abstract}) and include them with the
source files of your manuscript. Use the plain \texttt{abstract}
command if the abstract is also to appear in the printed version
of the book.}

\abstract{This chapter investigates both topological and
measure-theoretic directional entropies arising from semigroup
actions generated by one-dimensional linear cellular automata
(LCAs) and the shift transformation on the compact metric space
\(\mathbb{Z}_m^{\mathbb{Z}}\). It aims to present a collection of
formulas for computing these entropy measures. The work
consolidates and builds upon previous research in the field. The
study of toplological directional entropy (TDE) is carried out using Milnor’s framework, while the
measure-theoretic counterpart is analyzed through the
Kolmogorov–Sinai formalism.}

\section{Short introduction}
Both measure-theoretic entropy and topological entropy play fundamental roles
in dynamical systems theory and ergodic theory \cite{KKW-2024,Walters1982}.
Originally introduced by Shannon, entropy has become a crucial
research topic across multiple scientific disciplines. The
measure-theoretic approach to entropy, primarily developed by the
Russian school, including Kolmogorov and Sinai, has led to various
terminologies and methodologies for computing entropy in dynamical
systems \cite{Kolmogorov-1959, Sinai-1959}. In contrast, the
concept of topological entropy, first introduced by Adler et al.
\cite{Adler-1965}, has been progressively extended to continuous
functions on diverse spaces \cite{Walters1982}.

In 1986, Milnor introduced the concept of directional entropy
\cite{Milnor1986,Milnor1988}. Park made significant contributions to this
concept \cite{Park-1995,Park-1999}. Later, Kaminski and Park
extended the notion by defining directional entropy for
\(\mathbb{Z}^{2}\) actions on a Lebesgue space
\cite{Kaminski1999}. Courbage and Kaminski further explored the
metric directional entropy of \(\mathbb{Z}^{2}\)-actions,
particularly those arising from the natural extension of a
cellular automaton (CA). This extension operates alongside the
shift transformation on the space of doubly infinite sequences
over a finite state space \cite{Courbage2002}.

Recently, the topological entropy dimension of
$\mathbb{Z}^2$-topological dynamical systems has been studied. The
relationships between TDE and directional topological entropy
dimension were investigated \cite{liu2022directional}.  Wei et al.
\cite{WXZ-CMP-2024} investigate directional entropy along
non-rational directions within the two-dimensional space
$\mathbb{R}^2$, and they elucidate the configuration of
directional Pinsker $\sigma$-algebra within dynamical systems
preserving measures in $\mathbb{Z}^2$. In Ref.
\cite{BCR-IJM-2016}, Broderick, Cyr and Kra examine the
directional entropy of a dynamic system linked to a $\mathbb{Z}^2$
configuration within a finite set of symbols. In a recent study by
the author, detailed analyses of both measure-theoretic and
topological entropy for 1D CAs on the
ring \(\mathbb{Z}_m\) were presented \cite{Akin-2024}. In the
current work, we focus on the concept of directional entropy for
the \(\mathbb{Z}^2\)-action generated by 2D CAs. Our findings reveal that entropy values exhibit
directional variations. For further insights, we suggest
consulting the book chapter in \cite{Akin-2024}.

In this chapter, we examine the directional measure-theoretic and topological entropies associated with the $\mathbb{Z}^{2}$-action, formed by the interplay between a 1D CA and a shift transformation. Building upon established research, we offer a comprehensive analysis of the topic, enriched with numerous examples and illustrative figures to facilitate comprehension.

This chapter specifically investigates the directional entropy properties of $\mathbb{Z} \times \mathbb{Z}_+$-semigroup actions generated by a 1D CA $F$ and the shift map $\sigma$ on the full shift space $\mathbb{Z}_m^{\mathbb{Z}}$. The study employs two complementary frameworks of entropy theory.

\begin{itemize}
    \item \textbf{Topological entropy} (following Milnor's axiomatic formulation):
    \[
    h_{\mathrm{top}}(F, \sigma) = \lim_{n \to \infty} \frac{1}{n} \log N(n, \epsilon),
    \]
    where \(N(n, \epsilon)\) counts \((\epsilon,n)\)-separated orbits.
    
    \item \textbf{Measure-theoretic entropy} (via Kolmogorov-Sinai theory):
    \[
    h_\mu(F) = \sup_{\mathscr{P}} \lim_{n \to \infty} \frac{1}{n} H_\mu\left(\bigvee_{k=0}^{n-1} F^{-k}\mathscr{P}\right),
    \]
    where \(\mathscr{P}\) denotes a measurable partition of \(\mathbb{Z}_m^{\mathbb{Z}}\), and \(\vee\) is the join operation of partitions:
    \[
    \mathscr{A} \vee \mathscr{B} = \{A \cap B \mid A \in \mathscr{A},\ B \in \mathscr{B}\}.
    \]
\end{itemize}

\begin{itemize}
    \item \textbf{LCA Entropy Formula}: For linear cellular automata,
    \[
    h_{\mathrm{top}}(F, \sigma) = \log \lambda_{\max}(F),
    \]
    where \(\lambda_{\max}(F)\) is the maximal Lyapunov exponent of \(F\).
    
    \item \textbf{Markov Measure Entropy}: For stationary Markov measures \(\mu = \mu_{\pi P}\) with transition matrix \(P\) and stationary distribution \(\pi\),
    \[
    h_\mu(\sigma) = -\sum_{i,j} \pi_i P_{ij} \log P_{ij}.
    \]
\end{itemize}

\section{The MTDE of $\mathbb{Z}^2$-actions over the ring $\mathbb{Z}_{m}$}\label{Directional-entropy}

In this section,  we investigate the  MTDEs of
$\mathbb{Z}^2$-actions generated by 1D LCA and the shift map. Let
us give some necessary definitions.
\begin{definition}\cite{Favati1997}\label{left-permutive}\
Let $T_{f[l, r]}: \mathbb{Z}_m^\mathbb{Z}\to
\mathbb{Z}_m^\mathbb{Z}$ be a 1D- LCA associated to the local rule
$f$  with minimal memory $Mem[l,r]$. The local rule $f$ and the
1D- LCA $T_{f[l, r]}$ are called:
\begin{enumerate}
    \item[(1)] \emph{leftmost} \emph{permutive} if $l < 0$ and $f$ is permutative in $x_l$;
    \item[(2)] \emph{rightmost} \emph{permutive} if $r > 0$ and $f$ is permutative in $x_r$;
    \item[(3)] \emph{bilaterally permutative} (bipermutative) if it is both leftmost and rightmost permutative.
\end{enumerate}
\end{definition}

The investigation of invariant measures, a cornerstone of ergodic
theory, has been rigorously explored through diverse
methodological lenses \cite{Sablik2007}. Notably, Chang and Chen
\cite{CC-2013} established explicit formulas for both
measure-theoretic and topological entropy of weakly permutive CAs
under invariant measures on the configuration space
\(\mathbb{Z}_m^{\mathbb{Z}}\). A critical property of 1D CAs
\(F_{f[l,r]}\) is their commutativity with the shift map
\(\sigma\) \cite{Hedlund1969}, necessitating the characterization
of invariant measures for the semigroup action generated by
\(F_{f[l,r]}\) and \(\sigma\).

Let \(\mathcal{M}(\mathbb{Z}_m^{\mathbb{Z}})\) denote the space of
probability measures on \(\mathbb{Z}_m^{\mathbb{Z}}\) defined over
the \(\sigma\)-algebra \(\mathcal{B}\). A measure \(\mu \in
\mathcal{M}(\mathbb{Z}_m^{\mathbb{Z}})\) is termed
\(\sigma\)-invariant (respectively \(F_{f[l,r]}\)-invariant) if
\(\sigma\mu = \mu\) (respectively \(F_{f[l,r]}\mu = \mu\)).
Consequently, \(\mu\) is \((F_{f[l,r]}, \sigma)\)-invariant if it
is invariant under both transformations \cite{Sablik2007}.

For CAs governed by local rules dependent on finite neighborhoods,
\(F_{f[l,r]}\) exhibits invariance under spatial translation
\(\sigma\). This permits defining the \textbf{topological entropy}
of \(F_{f[l,r]}\) along direction \((p, q) \in \mathbb{Z} \times
\mathbb{N}\) as the entropy of \(F_{f[l,r]}^{q} \circ \sigma^{p}\)
\cite{Bernardo-Co-2001}. Consider a normalized ergodic measure
\(\mu\) invariant under the \(\mathbb{Z}^2\)-action. Let \(h_{m,n}
= h(\sigma^{m}F^{n}, \mu)\) denote the measure-theoretic entropy
of the transformation \(\sigma^{m}F^{n}\), where \(F =
F_{f[l,r]}\) is a 1D linear CA (LCA) over \(\mathbb{Z}_m\). This
subsection focuses on Bernoulli and Markov measures within the
measurable space \((\mathbb{Z}_m^{\mathbb{Z}}, \mathcal{B})\).

Define the partition \(\xi = \{\!_0[0],\!_0[1],
\ldots,\!_0[p-1]\}\) of \(\mathbb{F}_p^{\mathbb{Z}}\), and let
\[
\xi_{m,n} = \Phi^{(m,n)}\xi = \sigma^m F_{f[l,r]}^n(\xi),
\]
where \(\Phi^{(m,n)}\) represents the joint action of \(\sigma^m\)
and \(F_{f[l,r]}^n\).

To calculate the directional entropy, Sinai \cite{Sinai1985}
introduced standard notation and fundamental concepts from the
theory of measurable partitions and measure-theoretic entropy. The
segment on the plane connecting the points \((a, 0)\) and \((a +
\omega^{-l}, 1)\) is denoted by \(I = I(a, \omega)\), while
\(\Gamma(a, \omega)\) represents the half-line given by \(y =
\omega(x-a)\) for \(y \leq 1\). Clearly, \(I(a, \omega)\) is
contained within \(\Gamma(a, \omega)\). Throughout, we will assume
\(\omega > 0\).
\begin{figure}[!htbp]
\centering
\includegraphics[width=60mm]{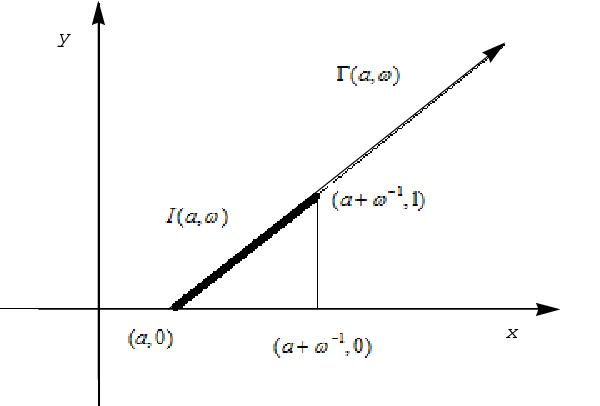}
\caption{The segment $I(a,\omega)$ on the plane joining the points
$(a, 0)$ and $(a +\omega^{-l}, 1)$.}\label{DE-segment}
\end{figure}
\begin{definition} Let $(X,\mathcal{A},\mu)$ be a measure space. Let $\alpha$ and $\beta$ be two partitions of $X$. The quantity
$$
H_{\mu}(\alpha|\beta)=-\sum\limits _{B\in \beta}\mu(B)\sum\limits
_{A\in \alpha}\mu(A|B)\log \mu(A|B)
$$
is called the conditional entropy of the partition $\alpha$ given
$\beta$.
\end{definition}
To prove the theorems, Sinai \cite{Sinai1985} deals with the
following conditional entropies:
$$
\begin{array}{l}
 \mathcal{H}_r(I)=H\left(\underset{m\geq a+\omega ^{-1}}{\bigvee }\xi _{m,1}|\underset{n=0}{\overset{\infty }{\bigvee }}
 \underset{m\geq a+\omega ^{-1}n}{\bigvee }\xi _{m,-n}\right) \\
 \mathcal{H}_l(I)=H\left(\underset{m\leq a+\omega ^{-1}}{\bigvee }\xi _{m,1}|\underset{n=0}{\overset{\infty }{\bigvee }}
 \underset{m\leq a+\omega ^{-1}n}{\bigvee }\xi _{m,-n}\right)\\
 \mathcal{H}(I) = \mathcal{H}_l(I) +\mathcal{H}_r(I).
\end{array}
$$

Let us consider a transformation $Q$ in the space of segments
$I(a, \omega)$, where $Q(I(a,\omega)) = I(a',to)$, $a'= a
+\omega^{-1}$. Our first result is the following theorem.

\begin{theorem}\cite[Theorem 1]{Sinai1985} Let $p>0, q>0$ have no common factor. Then
$$
h_{p,q}=\sum _{i=0}^{p-1} \mathcal{H}(Q^i(I))=p\int _0^1
\mathcal{H}(I)da
$$
for any interval $I = I(a,-q/p)$.
\end{theorem}

We recall the definitions and properties of the MTDE for
\(\mathbb{Z}^2\)-actions, using elementary entropy properties (see
\cite{Kaminski1999,Park-1995,Park-1999,Sinai-1959} for details).
Let \((X, \mathcal{B}, \mu)\) be a Lebesgue probability space, and
let \(\mathcal{Z}\) be the set of all countable measurable
partitions of \(X\) with finite entropy, equipped with the Rokhlin
metric:
\[
\varrho(P, Q) = H(P|Q) + H(Q|P),
\]
where \(H(P|Q)\) denotes the conditional entropy of \(P\) given
\(Q\), for \(P, Q \in \mathcal{Z}\). Consider a
\(\mathbb{Z}^2\)-action \(\Phi\) on \((X, \mathcal{B}, \mu)\). For
a set \(A \subset \mathbb{R}^2\) and \(P \in \mathcal{Z}\), we
define:
\[
P(A) = \bigvee_{(s,t) \in A \cap \mathbb{Z}^2} \Phi^{(s,t)} P.
\]
Let \(\vec{v} = (x, y)\) be a fixed vector in \(\mathbb{R}^2\),
and let \(\Gamma\) be the family of bounded subsets of
\(\mathbb{R}^2\). Let \((T, S)\) be an ordered pair of commuting
automorphisms of \(X\) generating \(\Phi\), i.e.,
\[
\Phi^{(m,n)} = T^m \circ S^n, \quad (m,n) \in \mathbb{Z}^2.
\]
For any partition \(P \in \mathcal{Z}\), the directional entropy
is given by:
\[
h_{\vec{v}}((T,S), P) = \sup_{B \in \Gamma} \limsup_{t \to \infty}
\frac{1}{t} H(P(B + [0,t)\vec{v})).
\]
We define the directional mean entropy of the
\(\mathbb{Z}^2\)-action \(\Phi\) with respect to \(P\) in the
direction \(\vec{v}\) as:
\[
h_{\vec{v}}(\Phi, P) = h_{\vec{v}}((T_0, S_0), P),
\]
where \(T_0 = \Phi^{(1,0)}\) and \(S_0 = \Phi^{(0,1)}\).

It is shown that:
\[
h_{\vec{v}}(\Phi, P) = \lim_{m \to \infty} \lim_{t \to \infty}
\frac{1}{t} H(P(R(\vec{v}, m, t))),
\]
where \(R(\vec{v}, m, t)\) is defined as:
\[
R(\vec{v}, m, t) = \left\{
\begin{array}{ll}
\{(i,j) \in \mathbb{Z}^2; 0 \leq j \leq [t y], -m + j \frac{x}{y} < i \leq m + j \frac{x}{y}\}, & \text{if } y \neq 0, \\
\{(i,j) \in \mathbb{Z}^2; -m < j \leq m, 0 \leq i \leq [t x]\}, &
\text{if } y = 0.
\end{array}
\right.
\]

\begin{definition}\label{directional-entropy1}
The quantity $h_{\vec{v}}\left( \Phi \right) =\sup\limits_{P
\in\mathcal{Z}}h_{\vec{v}}\left( \Phi, P \right)$ is called the
directional entropy of $\Phi$ in the direction $\vec{v}$, where
$\mathcal{Z}$ is the set of all countable measurable partitions of
$\mathbb{Z}_{p^{2}}^{\mathbb{Z}}$ with finite entropy (see
\cite{Courbage2002,Kaminski1999} for more details).
\end{definition}
Here, we consider $T_0=T_{f[l,r]}$ and $T_0=\sigma$. In this case,
from the definition of entropy, we get
\begin{itemize}
\item[(i)] for every $\alpha \in \mathbb{R}$
\begin{equation}\label{eq11}
h_{\alpha\vec{v}}(\Phi) = |\alpha| h_{\vec{v}}(\Phi).
\end{equation}
\item[(ii)]  If $\vec{v}=(s,t)\in \mathbb{Z}^2$ then
\begin{equation}\label{11eq12}
 h_{\vec{v}}(\Phi)=
h(\Phi^{(s,t)})= h(\sigma^{s}T_{f[l,r]}^{t}).
\end{equation}
\end{itemize}
Let $(x,y)\in \mathbb{R}^{2}$, denote $z_l=x+ly,z_r$ or
$z_r=x+ry.$ In \cite{Courbage2002}, Courbage and Kaminski have
proved the following results. They obtained exact formulas for
directional entropy of CA-action $\Phi$ satisfying invariant Borel
probability measure.

\begin{proposition}\cite[Proposition]{Courbage2002}
For any CA-action $\Phi$ on the compact metric space
$\mathbb{Z}_m^{\mathbb{Z}^{2}}$ and any $\Phi$-invariant Borel
probability measure we have
\begin{equation}\label{measure-direc-ent1aa}
h(\theta)=\left\{
\begin{array}{ll}
 \max (\left|z_l\right|,\left|z_r\right|)\log m, &if\ z_r z_l\geq 0, \\
 \left|z_r-z_l\right|\log m, & if\ z_r.z_l\leq 0.
\end{array}
\right.
\end{equation}
\end{proposition}
We can summarize the results obtained by Courbage and Kaminski
\cite{Courbage2002} as follows.

\begin{theorem}
\label{Thm-MTDE1q} Let $l, r\in \mathbb{Z}, l\leq r$, be given and
let $f : \mathbb{Z}_m^{r-l+1}\rightarrow \mathbb{Z}_m$ be a fixed
local rule. Then we have \begin{enumerate} \item[(1)]  If the
local rule $f$ is right permutative then
$$
h_{\vec{v}}(\Phi)=|z_r| \log m
$$ for all $\vec{v} =
(x, y)$ with $0\leq z_l\leq z_r$ or $z_r\leq z_l\leq 0.$
\item[(2)] If the local rule $f[l,r]$ is left permutative then
$$h_{\vec{v}}(\Phi) = |z_l | \log m $$
for all $\vec{v} =(x, y)$ with $ z_l\leq z_r\leq 0$ or $0\leq
z_r\leq z_l.$ \item[(3)] If the local rule $f$ is bipermutative
and $z_l.z_r\leq 0$, then
$$
h_{\vec{v}}(\Phi)= |z_r - z_l | \log m.
$$
\item[(4)] If $\Phi$ and $\Psi$ are $\mathbb{Z}^{2}$-actions and
$\Psi$ is a factor of $\Phi$ then $h_{\vec{v}}(\Phi)\geq
h_{\vec{v}}(\Psi)$.
    \item[(5)]If $\mathrm{P}\in \mathcal{Z}$  is
generator $\Phi$ then
\begin{equation}\label{11eq12a}
 h_{\vec{v}}(\Phi)=
h(\Phi,\mathrm{P}).
\end{equation}
\end{enumerate}
\end{theorem}
We will not give proof of these results here, so we recommend that
the readers refer to the references
\cite{Courbage2002,Kaminski1999}.

\subsection{Examples}

For the examples we will give here, we will consider the unit
circle $\mathbb{S}^{1}=\{(x,y)\in \mathbb{R}^{2}:x^{2}+y^{2}=1\}$
instead of $\mathbb{R}^{2}$, in other words, we will calculate the
directional entropy of the function $\Phi$ in direction
$\vec{v}\in\mathbb{S}^{1}$.

\begin{figure}[!htbp]
\centering
\includegraphics[width=60mm]{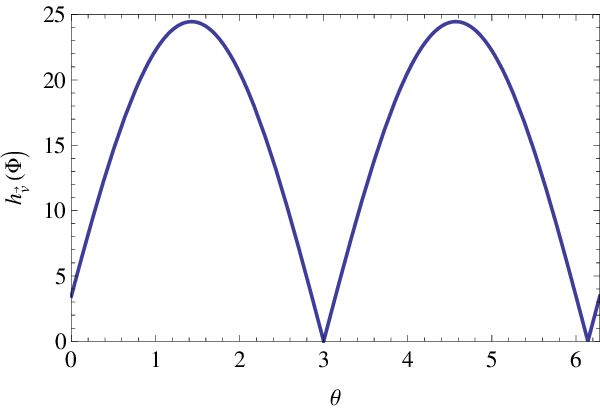}
\caption{The graph of the entropy function given in
\eqref{exam-MTDE-eq1}.}\label{MTDE1}
\end{figure}

\begin{example}\label{ex:lP-p=11} Let us consider the local rule as
$$
f(x_0,x_1,\cdots ,x_5)=2x_0+4x_1+3x_2+x_3+6x_4+7x_5\pmod {11}.
$$
From the definition \ref{left-permutive} (2), $f$ is the rightmost
permutative. Let $\vec{v}\in \mathbb{S}^{1}$. Then we have $z_r=x
+7y=\cos \theta +7\sin \theta$. From Theorem \ref{Thm-MTDE1q} (1),
we get
\begin{equation}\label{exam-MTDE-eq1}
h_{\vec{v}}(\Phi )=|z_r|\log 11=|\cos \theta +7\sin \theta |\log
11.
\end{equation}
Fig. \ref{MTDE1} shows the graph of function $h_{\vec{v}}(\Phi )$
given in \eqref{exam-MTDE-eq1} in the interval $[0,2\pi]$.
\end{example}
 
\begin{example}\label{ex:rP-p=19}  Let us consider the local rule as
$$
f(x_{-3},x_{-2},x_{-1} ,x_0)=6x_{-3}+3x_{-2}+5x_{-1}+2x_0\pmod
{19}.
$$
From the definition \ref{left-permutive} (1), $f$ is the leftmost
permutative. Let $\vec{v}\in \mathbb{S}^{1}$. Then we have $z_l=x
-3y=\cos \theta -3\sin \theta$. From Theorem 18 (1), we get
\begin{equation}\label{exam-MTDE-eq2}
h_{\vec{v}}(\Phi )=|z_l|\log 19=|x -3y |\log 19=|\cos \theta
-3\sin \theta |\log 19.
\end{equation}

\begin{figure}[!htbp]
\centering
\includegraphics[width=60mm]{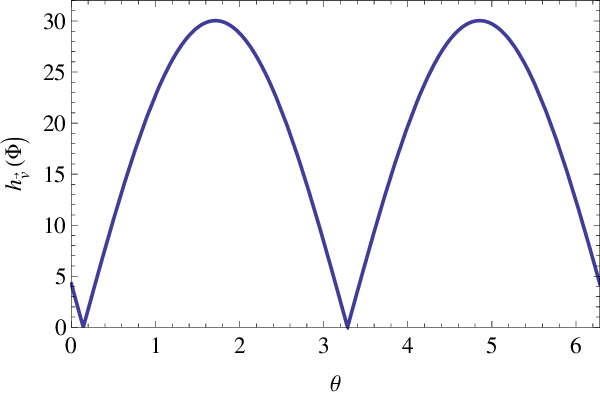}
\caption{The graph of the entropy function given in
\eqref{exam-MTDE-eq2}.}\label{MTDE2}
\end{figure}

Fig. \ref{MTDE2} shows the graph of function $h_{\vec{v}}(\Phi )$
given in \eqref{exam-MTDE-eq2} in the interval $[0,2\pi]$.
\end{example}

\begin{example}\label{ex:rbP-p=23} Let us consider the local rule as
$$f(x_{-2},x_{-1},x_0,x_1,x_2,x_3)=12x_{-2}+3x_{-1}+5x_0+4x_1+x_2+21x_3\pmod {23}.
$$
It is clear that the local rule $f$ is the bipermutative. Assume
that $\vec{v}\in \mathbb{S}^{1}$. Then, we get
$$
z_l=x-2y=\cos \theta -2\sin \theta,\ z_r=x+3y=\cos \theta+3\sin
\theta.
$$
\begin{figure}[!htbp]\label{DE-segment33}
\centering
\includegraphics[width=95mm]{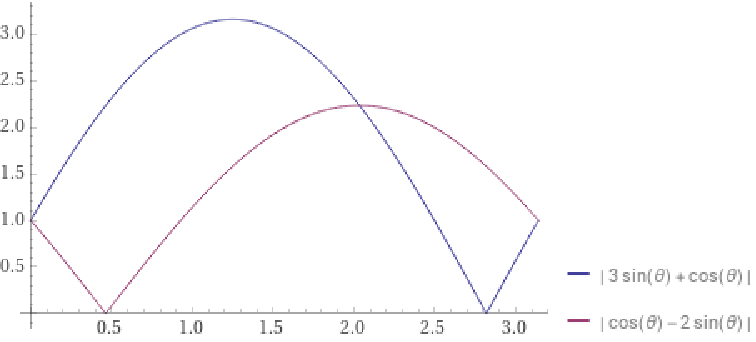}
\caption{The graphs of the functions $|z_l|$ and $|z_r|$
associated with the local rule
\eqref{exam-MTDE-eq2}.}\label{zlandzr}
\end{figure}
Now let us determine the intervals where the function $z_lz_r$ is
positive and negative. Consider
\begin{equation*}
A(\theta)=(\cos \theta -2\sin \theta)(\cos \theta+3\sin \theta).
\end{equation*}

From some elementary operations, one can show that $A(\theta)<0$
for $\theta\in(arccot(2),arccot(-3))$,  and $A(\theta)\geq 0$ for
$\theta\in[ 0,arccot(2)]\cup[arccot(-3),\pi]$. From the equation
\eqref{measure-direc-ent1aa}, we get
\begin{equation}\label{MTE-EXAM3a}
h_{\vec{v}}(\Phi) = \left\{
\begin{array}{ll}
|\cos \theta + 3 \sin \theta| \log 23, & \text{for } 0 \leq \theta \leq 0.46 \\
5 |\sin \theta| \log 23, & \text{for } 0.46 < \theta < 2.67795 \\
|\cos \theta - 2 \sin \theta| \log 23, & \text{for } 2.67795\leq \theta \leq \pi.
\end{array}
\right.
\end{equation}
\begin{figure}[!htbp]
\centering
\includegraphics[width=75mm]{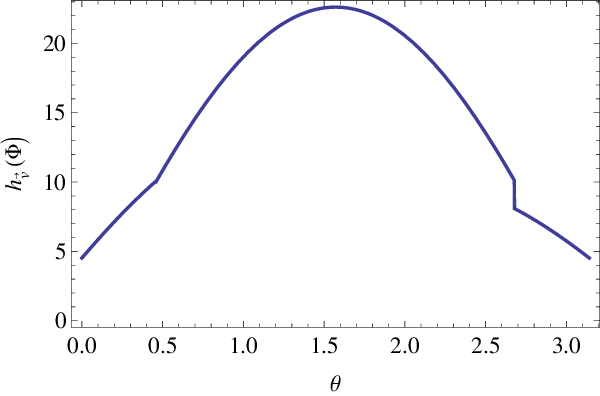}
\caption{The graph of function  $h_{\vec{v}}(\Phi )$ given in
\eqref{MTE-EXAM3a} in the interval $[0,\pi]$.}\label{Plot-MTE3}
\end{figure}
\end{example}

The examples presented in this section include one case each illustrating a left permutative local rule (Example~\ref{ex:lP-p=11}), a right permutative local rule (Example~\ref{ex:rP-p=19}), and a bipermutative local rule (Example~\ref{ex:rbP-p=23}).

\section{MTDE for Bernoulli and Markov measures}\label{sec:MTDE-B-M}
In this section, we derive upper bounds for the MTDE by applying the Bernoulli and Markov measures specifically to the family of cylindrical sets.

\subsection{The MTDE with respect to Bernoulli measure}

This subsection analyzes the MTDE of a \(\mathbb{Z}^2\)-action
generated by an invertible 1D LCA and the shift map on
\(\mathbb{Z}^{\mathbb{Z}}_{p^k}\). Unlike the method in
\cite{Courbage2002}, which uses the natural extension, we
establish an upper bound for the directional entropy w.r.t any
Bernoulli measure.

From the equation \eqref{eq11}, for all
$v=(x,y)=\sqrt{x^2+y^2}\left(\frac{x}{\sqrt{x^2+y^2}},\frac{y}{\sqrt{x^2+y^2}}\right)\in
\mathbb{R}^2$, we have
\begin{align}\label{MTDE-Sinai-1}
h_{\vec{v}}(\Phi
)=\sqrt{x^2+y^2}h_{(\frac{x}{\sqrt{x^2+y^2}},\frac{y}{\sqrt{x^2+y^2}})}(\Phi
). 
\end{align}
Let us consider a thin cylinder set
$$
C = _{a}[j_{0},\ j_{1},\cdots, j_{s}]_{s+a}=\{x\in
\mathbb{Z}_{m}^{\mathbb{Z}}:x_a=j_{0},\ldots, x_{a+s}=j_{s}\},
$$
where $j_0,\ j_1,\cdots, j_s \in \mathbb{Z}_{m}$.

In this subsection, one of our main results is the following
result.

\begin{theorem}\cite{AkinIJMPC-2011DE}\label{DirEnt-Ber-1}
Let \( f: \mathbb{Z}_{p^2}^{r-l+1} \rightarrow \mathbb{Z}_{p^2} \)
(\( p \) prime) be a linear local rule, and let \(\mu_{\pi}\) be
the Bernoulli measure associated with the probability vector \(\pi
= (p_0, p_1, \ldots, p_{p^2-1})\). For any direction \(\vec{v} =
(s, q) \in \mathbb{R}^2\) (assuming \( s \) is not large), the
directional entropy satisfies:
\[
0 \leq h_{\vec{v}}(\Phi) \leq -|q|(r-l) \sum_{i=0}^{p^2-1} p_i
\log p_i.
\]
\end{theorem}

\begin{proof}
If \( p \) divides all coefficients \(\lambda_i\) of \( f \) for
\( i \in \{l, \ldots, r\} \), then \( h_{\vec{v}}(\Phi) = 0 \).
Assume \(\gcd(p, \lambda_i) = 1\) for at least one \( i \). For
\(\vec{v} = (s, q) \in \mathbb{Z}^2\) with \( q \geq 0 \), define
the zero-time partition \(\xi = \{_{0}[0], _{0}[1], \ldots,
_{0}[p^2-1]\}\) and its refinement:
\[
\xi(-i, i) = \bigvee_{u=-i}^i \sigma^{-u} \xi.
\]
Then, we have
\[
\bigvee_{k=0}^n \Phi^{(-sk, -qk)} \xi(-i, i) = \bigvee_{k=0}^n
\sigma^{-sk} F^{-qk}_{f[l,r]} \xi(-i, i) \preceq \xi(-i + n(ql -
s), i + n(ql - s)).
\]

The entropy bound follows from:
\begin{align*}
&H_{\mu_{\pi}}\left(\bigvee_{k=0}^n \sigma^{-sk} F^{-qk}_{f[l,r]} \xi(-i, i)\right) \leq H_{\mu_{\pi}}(\xi(-i + n(ql - s), i + n(ql - s))) \\
&= -\sum_{j_{-i + n(ql-s)}, \ldots, j_{i + n(ql-s)}=0}^{p^2-1}
p_{j_{-i + n(ql-s)}} \cdots p_{j_{i + n(ql-s)}} \log
\left(p_{j_{-i + n(ql-s)}} \cdots p_{j_{i + n(ql-s)}}\right).
\end{align*}

By induction, this sum simplifies to:
\[
-\left(nq(r-l) + 1\right) \sum_{i=0}^{p^2-1} p_i \log p_i.
\]
Thus, the mean entropy satisfies:
\[
h_{\mu_{\pi}}(\xi(-i, i), \sigma^s F^{q}_{f[l,r]}) \leq -q(r-l)
\sum_{i=0}^{p^2-1} p_i \log p_i.
\]
As \(\xi(-i, i)\) refines to the partition, we obtain:
\[
h_{\vec{v}}(\Phi) = h_{\mu_{\pi}}(\sigma^s F_{f[-l,r]}^{q}) \leq
-q(r-l) \sum_{i=0}^{p^2-1} p_i \log p_i.
\]

For \( q = -h < 0 \), the inverted LCA \( F_{g[-(2r-l), -r]}^h
\circ \sigma^s \) yields:
\[
h_{\mu_{\pi}}(\xi(-i, i), \sigma^s F^{h}_{g[-(2r-l), -r]}) \leq
-h(l-r) \sum_{i=0}^{p^2-1} p_i \log p_i.
\]
\end{proof}
Now let us present the following example, taking Theorem
\ref{DirEnt-Ber-1} into account:
\begin{example}\label{ex:MTDE-m=4-B} \cite{AkinIJMPC-2011DE}
Consider the linear local rule \[f(x_{-1}, x_0, x_1) = 2x_{-1} +
2x_0 + 3x_1 \pmod{4}. \] Its finite power series (FPS) is:
\[
U(X) = 2X + 2 + 3X^{-1} = 3X^{-1} \left(1 + 2X + 2X^2\right)
\pmod{4}.
\]
The inverse FPS \( S(X) \), computed modulo 4, is:
\[
S(X) = 3X \left(1 - 2X - 2X^2\right) = 3X - 6X^3 - 6X^4 \equiv 3X
+ 2X^2 + 2X^3 \pmod{4}.
\]
This corresponds to the local rule:
\[
g(x_{-3}, x_{-2}, x_{-1}) = 2x_{-3} + 2x_{-2} + 3x_{-1} \pmod{4}.
\]

Let \(\mu_{\pi}\) be the Bernoulli measure with probability vector
\(\pi = \left(\frac{1}{2}, \frac{1}{8}, \frac{1}{8},
\frac{1}{4}\right)\). For any direction \(\vec{v} = (s, q) \in
\mathbb{R}^2\), the directional entropy satisfies:
\[
h_{\vec{v}}(\Phi) \leq \frac{7}{2} |q| \log 2.
\]
\end{example}



\subsection{The MTDE with respect to Markov measure} \label{subsec:markov-measure-entropy}

In this subsection, we analyze the directional entropy of a
\(\mathbb{Z}^2\)-action \(\Phi\) generated by the shift
transformation \(\sigma\) and an invertible 1D linear CA with local rule \(f: \mathbb{Z}_{p^k}^{r-l+1} \to
\mathbb{Z}_{p^k}\) over the ring \(\mathbb{Z}_{p^k}\) (\(p\)
prime, \(k \geq 2\)), under the Markov measure \(\mu_{\pi T}\).
The analysis excludes considerations of natural extensions.

To compute the \textbf{MTDE}, we require:
\begin{itemize}
\item \textbf{Stochastic Transition Matrix}: A matrix \(T = (t_{ij})\) of size \(p^k \times p^k\), where \(t_{ij}\) denotes the transition probability from state \(i\) to \(j\), satisfying:
    \[
    \sum_{j=0}^{p^k-1} t_{ij} = 1 \quad \forall i \in \{0, \ldots, p^k - 1\}.
    \]
\item \textbf{Stationary Distribution}: A probability row vector \(\pi = (\pi_0, \pi_1, \ldots, \pi_{p^k-1})\) satisfying:
    \[
    \pi T = \pi \quad \text{and} \quad \sum_{i=0}^{p^k-1} \pi_i = 1.
    \]
\end{itemize}

The pair \((\pi, T)\) defines the Markov measure \(\mu_{\pi T}\) \cite{Denker1976}.
The main result of this subsection is the following theorem.
\begin{theorem} \label{thm:dir-entropy-bound} (\cite{AkinIJBC-2012DE})
Let \(f\) be the linear local rule \(f: \mathbb{Z}_{p^k}^{r-l+1}
\to \mathbb{Z}_{p^k}\). For any direction \(\vec{v} = (s, q) \in
\mathbb{R}^2\), the MTDE of the \(\mathbb{Z}^2\)-action \(\Phi\)
satisfies:
\[
0 \leq h_{\vec{v}}(\Phi, \mu_{\pi T}) \leq -|q|(r-l)
\sum_{i,j=0}^{p^k - 1} \pi_i t_{ij} \log t_{ij},
\]
where \(t_{ij}\) are entries of the stochastic matrix \(T\), and
\(s\) is assumed not to be large.
\end{theorem}

\begin{proof}
If \(p\) divides all coefficients \(\lambda_i\) of \(f\) for \(i
\in \{l, \ldots, r\}\), then \(h_{\vec{v}}(\Phi) = 0\). Assume
\(\gcd(p, \lambda_i) = 1\) for at least one \(i\). For \(\vec{v} =
(s, q) \in \mathbb{Z}^2\) with \(q \geq 0\), define the zero-time
partition \(\xi = \{_{0}[0], \ldots, _{0}[p^k - 1]\}\). Then:
\[
\bigvee_{k=0}^n \sigma^{-sk} F^{-qk}_{f[l,r]} \xi(-i, i) \preceq
\xi(-i + n(ql - s), i + n(ql - s)).
\]
The entropy bound follows from:
\begin{align}\label{eq:entropy-estimate}
H_{\mu_{\pi T}}\left(\bigvee_{k=0}^n\sigma^{-sk}F^{-qk}_{f[l,r]}\xi\right) &\leq H_{\mu_{\pi T}}(\xi(n(ql-s),n(qr-s))) \nonumber \\
&= -\sum_{\substack{j_{n(ql-s)}, \ldots, \\ j_{n(qr-s)}=0}}^{p^k-1} \pi_{j_{n(ql-s)}} t_{j_{n(ql-s)}j_{n(ql-s)+1}} \cdots t_{j_{n(qr-s)-1}j_{n(qr-s)}} \nonumber \\
&\quad \times \log\left(\pi_{j_{n(ql-s)}}
t_{j_{n(ql-s)}j_{n(ql-s)+1}} \cdots
t_{j_{n(qr-s)-1}j_{n(qr-s)}}\right).
\end{align}

By induction, the sum in \eqref{eq:entropy-estimate} simplifies
to:
\[
-\left(nq(r-l) + 1\right) \sum_{i,j=0}^{p^k - 1} \pi_i t_{ij} \log
t_{ij}.
\]
Thus, the mean entropy satisfies:
\[
h_{\mu_{\pi T}}(\xi, \sigma^{s}F^{q}_{f[l,r]}) \leq -q(r-l)
\sum_{i,j=0}^{p^k - 1} \pi_i t_{ij} \log t_{ij}.
\]
As \(\xi(-i,i)\) refines to the partition, we obtain:
\[
h_{\vec{v}}(\Phi) = h_{\mu_{\pi T}}(\sigma^{s}F_{f[-l,r]}^{q})
\leq -q(r-l) \sum_{i,j=0}^{p^k - 1} \pi_i t_{ij} \log t_{ij}.
\]
For \(q = -h < 0\), analogous reasoning for the inverted LCA
\(F_{g[-(2r-l),-r]}\) yields:
\[
h_{\mu_{\pi T}}(\xi, \sigma^{s}F^{h}_{g[-(2r-l),-r]}) \leq -h(l-r)
\sum_{i,j=0}^{p^k - 1} \pi_i t_{ij} \log t_{ij}.
\]
\end{proof}

\subsubsection*{Remarks}
\begin{itemize}
    \item The Markov measure \(\mu_{\pi T}\) is generally not invariant under \(\Phi\).
    \item The bound decouples the directional scaling factor \(|q|(r-l)\) from the entropy rate of the Markov chain.
\end{itemize}

Now let us give the following examples.
\begin{example}\label{ex:MTDE-m=4-M} \label{ex:markov-entropy-bound} (\cite{AkinIJBC-2012DE})
Consider the local rule \( f(x_{-1}, x_0, x_1) = 2x_{-1} + 2x_0 +
3x_1 \pmod{4} \). The associated finite power series (FPS) is:
\[
U(X) = 2X + 2 + 3X^{-1} = 3X^{-1}\left(1 + 2X^1 + 2X^2\right)
\pmod{4}.
\]
The inverse FPS \( S(X) \) modulo 4 is computed as:
\[
S(X) = 3X\left(1 - 2X^1 - 2X^2\right) = 3X - 6X^2 - 6X^3 \equiv 3X
+ 2X^2 + 2X^3 \pmod{4}.
\]
This corresponds to the local rule:
\[
g(x_{-3}, x_{-2}, x_{-1}) = 2x_{-3} + 2x_{-2} + 3x_{-1} \pmod{4}.
\]

Given the stochastic matrix:
\[
T = \begin{pmatrix}
  \frac{1}{2} & \frac{1}{2} & 0 & 0 \\
  \frac{1}{8} & 0 & \frac{1}{8} & \frac{3}{4} \\
  0 & \frac{1}{16} & \frac{1}{16} & \frac{7}{8} \\
  0 & 0 & 1 & 0 \\
\end{pmatrix},
\]
solve \( \pi T = \pi \):
\begin{align}
  \frac{1}{2}\pi_0 + \frac{1}{8}\pi_1 &= \pi_0, \label{eq:pi0} \\
  \frac{1}{2}\pi_0 + \frac{1}{16}\pi_2 &= \pi_1, \label{eq:pi1} \\
  \frac{1}{8}\pi_1 + \frac{1}{16}\pi_2 + \pi_3 &= \pi_2, \label{eq:pi2} \\
  \frac{3}{4}\pi_1 + \frac{7}{8}\pi_2 &= \pi_3. \label{eq:pi3}
\end{align}

From \eqref{eq:pi0}: \( \pi_1 = 4\pi_0 \). From \eqref{eq:pi1}: \(
\pi_2 = 56\pi_0 \). From \eqref{eq:pi2} and \eqref{eq:pi3}: \(
\pi_3 = 52\pi_0 \). Normalizing \( \pi_0 + \pi_1 + \pi_2 + \pi_3 =
1 \):
\[
\pi = \left(\frac{1}{113},\, \frac{4}{113},\, \frac{56}{113},\,
\frac{52}{113}\right).
\]

\textbf{Directional Entropy Bound:}
The entropy rate \( h(\mu_{\pi T}) \) is computed as:
\[
h(\mu_{\pi T}) = -\sum_{i=0}^3 \pi_i \sum_{j=0}^3 t_{ij} \log
t_{ij} \approx 0.3502.
\]
By Theorem \ref{thm:dir-entropy-bound}, for any direction \(
\vec{v} = (s, q) \in \mathbb{R}^2 \), the directional entropy
satisfies:
\[
0 \leq h_{\vec{v}}(\Phi, \mu_{\pi T}) \leq 2|q| \cdot
\log\left(\frac{2^{1.638}}{3^{0.026} \cdot 7^{0.433}}\right)
\approx 0.5308|q|.
\]
This upper bound quantifies the asymptotic uncertainty rate along
\( \vec{v} \).
\end{example}
 
\begin{example}\label{ex:MTDE-m=9-M}  Let us consider local rule
\begin{align}\label{eq:LR-9}
f(x_{-1}, x_0, x_1)&= 4x_{-1} +3x_0 +3x_1 \pmod{3^{2}}.
\end{align} We obtain the finite power series \emph{fps} associated with $f$ as
$$
U(X) = 4X^{1}+3X^{0}+3X^{-1}= 4X^{1}[1+3(X^{-1}+X^{-2})].
$$
As a result, the inverse of $U$ may be found as follows:
$$
S(X) = 7X^{-1}[1+6(X^{-1}+X^{-2})].
$$
It is clear that $U(X)S(X)=1 \mod 9$. By some calculations, the
local rule assiciated with $S(X)$ is obtained as
$$g(x_{1}, x_{2}, x_{3})= 7x_{1}+ 6x_{2}+ 6x_{3} \pmod {9}.
$$
Now we compute the MTDE of the $\mathbb{Z}^2$-action generated by:
 the deterministic CA with local rule $f$ given in \eqref{eq:LR-9} and the shift map $\sigma$ with respect tothe Markov measure $\mu_{\pi T}$.

Let us consider the transition matrix \( T \), with all rows
summing to 1:
\[
T = \left[\begin{matrix}
\frac{1}{3} & 0 & \frac{1}{3} & 0 & 0 & 0 & \frac{1}{6} & \frac{1}{6} & 0\\
\frac{1}{18} & \frac{1}{18} & 0 & \frac{1}{3} & 0 & \frac{1}{3} & \frac{1}{9} & \frac{1}{9} & 0\\
0 & 0 & 0 & \frac{1}{2} & 0 & 0 & \frac{1}{4} & \frac{1}{4} & 0\\
\frac{1}{4} & 0 & \frac{1}{4} & 0 & \frac{1}{4} & 0 & 0 & \frac{1}{8} & \frac{1}{8}\\
\frac{1}{7} & \frac{1}{7} & \frac{1}{7} & 0 & \frac{1}{7} & 0 & \frac{1}{7} & \frac{1}{7} & \frac{1}{7}\\
0 & \frac{1}{6} & \frac{1}{6} & \frac{1}{12} & \frac{1}{12} & \frac{1}{12} & \frac{1}{12} & \frac{1}{6} & \frac{1}{6}\\
\frac{1}{5} & \frac{1}{5} & 0 & 0 & \frac{1}{5} & \frac{1}{10} & \frac{1}{10} & \frac{1}{10} & \frac{1}{10}\\
\frac{1}{18} & \frac{1}{18} & 0 & \frac{1}{9} & \frac{2}{9} & \frac{1}{3} & \frac{1}{9} & \frac{1}{18} & \frac{1}{18}\\
\frac{1}{5} & \frac{1}{5} & \frac{1}{10} & \frac{1}{10} & 0 &
\frac{1}{10} & \frac{1}{10} & \frac{1}{10} & \frac{1}{10}
\end{matrix}\right]
\]


We seek a probability column vector \( X  \) such that:
\[T \cdot X = X\]
Solving this eigenvalue problem and normalizing the result, after
some elementary operations we obtain the stationary distribution:
\[X= \left[\begin{matrix}0.1441\\0.0836\\0.1066\\0.1269\\0.0949\\0.1155\\0.1175\\0.1354\\0.0755\end{matrix}\right]\]
So, one obtains the stationary distribution vector \(\pi\) as:
\[
\pi =X^\top= \begin{bmatrix} 0.1441 & 0.0836 & 0.1066 & 0.1269 &
0.0949 & 0.1155 & 0.1175 & 0.1354 & 0.0755 \end{bmatrix}
\]

The entropy rate \( h(\mu_{\pi T}) \) is computed as:
\[
h(\mu_{\pi T}) = -\sum_{i=0}^8 \pi_i \sum_{j=0}^8 t_{ij} \ln
t_{ij}
\]

Let us compute for $i=0,1,\cdots,8$ the row-wise entropies
\[H_i=\sum_{j=0}^8 t_{ij} \ln t_{ij}.\]

\[
\begin{array}{|c|c|c|c|}
\hline
\text{State } i & \text{Entropy } H_i & \pi_i & \pi_i H_i \\
\hline
0 & -\left( \frac{2}{3}\ln{\frac{1}{3}} + \frac{1}{6}\ln{\frac{1}{6}} + \frac{1}{6}\ln{\frac{1}{6}} \right) \approx 1.3297 & 0.1441 & 0.1913 \\
1 & -\left(\frac{2}{18}\ln{\frac{1}{18}} + \frac{2}{3}\ln{\frac{1}{3}} + \frac{2}{9}\ln{\frac{1}{9}} \right) \approx 1.5418 & 0.0836 & 0.1288 \\
2 & -\left( \frac{1}{2}\ln{\frac{1}{2}} + 2 \cdot \frac{1}{4}\ln{\frac{1}{4}} \right) \approx 1.0397 & 0.1066 & 0.1108 \\
3 & -\left( 3 \cdot \frac{1}{4}\ln{\frac{1}{4}} + \cdot \frac{2}{8}\ln{\frac{1}{8}} \right) \approx 1.5595 & 0.1269 & 0.1976 \\
4 & -\left( 7 \cdot \frac{1}{7}\ln{\frac{1}{7}} \right) = \ln{7} \approx 1.9459 & 0.0949 & 0.1845 \\
5 & -\left(\frac{4}{6}\ln{\frac{1}{6}} + \frac{4}{12}\ln{\frac{1}{12}} \right) \approx 2.0228 & 0.1155 & 0.2336 \\
6 & -\left(\frac{3}{5}\ln{\frac{1}{5}} + 4 \cdot \frac{1}{10}\ln{\frac{1}{10}} \right) \approx 1.8866 & 0.1175 & 0.2216 \\
7 & -\left(\frac{4}{18}\ln{\frac{1}{18}} + \frac{4}{9}\ln{\frac{1}{9}} + \frac{1}{3}\ln{\frac{1}{3}} \right) \approx 1.8309 & 0.1354 & 0.2480 \\
8 & -\left(\frac{2}{5}\ln{\frac{1}{5}} + \frac{5}{10}\ln{\frac{1}{10}} \right) \approx 1.7951 & 0.0755 & 0.1355 \\
\hline
\multicolumn{3}{|r|}{\mathbf{Total\ } h(\mu_{\pi T})} & \mathbf{1.6517} \\
\hline
\end{array}
\]

For any direction vector \((a, b) \in \mathbb{Z}^2\), the
directional entropy is:
\[
h_{(a, b)}(\mu_{\pi T}) = |b| \cdot h(\mu_{\pi T}) = |b| \cdot
1.6517.
\]
Examples:
\begin{itemize}
    \item Along vector \((0, 1)\): \( h_{(0, 1)} = 1.6517\),
    \item Along vector \((2, 3)\): \( h_{(2, 3)} = 3 \times 1.6517 = 4.9551\).
\end{itemize}
\end{example}

\paragraph{Methodological Note}  
Throughout this section, the examples of 1D-CAs—specifically  
\ref{ex:MTDE-m=4-B}, \ref{ex:MTDE-m=4-M}, and \ref{ex:MTDE-m=9-M}—are deliberately constructed to satisfy the property of \textit{reversibility}. This ensures that every configuration within the automaton’s state space admits a unique predecessor under its evolution rules, preserving invertibility in both topological and measure-theoretic frameworks..  

\subsection{Comments and References} \label{sec:comments-refs}
The findings in Section~\ref{Directional-entropy} build upon
foundational work in the entropy theory of dynamical systems,
extending and integrating results from key studies such
as~\cite{Akin-ergodic-2005, Akin-AMC-2005a,AkinIJMPC-2011DE,
AkinIJBC-2012DE, Courbage2002,Kaminski1999, Milnor1986,
Milnor1988, Park1994,Park-1995, Sinai1985}. Notable contributions
that serve as a basis for this development include:
\begin{itemize}
    \item \textbf{Milnor's Contribution}: The foundational concept of directional entropy in \(\mathbb{Z}^2\)-actions was introduced by Milnor~\cite{Milnor1986}, laying the groundwork for analyzing entropy along specific spatial directions.
    \item \textbf{Entropy Continuity}: Park~\cite{Park1994} established that for systems generated by CAs, directional entropy varies continuously with direction vectors \(\vec{v} \in \mathbb{R}^2\), under invariant Borel probability measures.
\end{itemize}
In earlier work, Akin~\cite{Akin-AMC-2005a} derived explicit
expressions for MTDE associated with \(\mathbb{Z} \times
\mathbb{Z}^+\)-actions driven by shift transformations and
additive one-dimensional CA. However, extending such calculations
to more general settings remains a complex task due to:
\begin{itemize}
    \item The analytical challenges of harmonizing measure-theoretic and topological entropy frameworks,
    \item The intricate combinatorial structures arising from multidimensional state dependencies.
\end{itemize}
The present section further advances this line of inquiry by
exploring directional entropy in \(\mathbb{Z}^2\)-actions under
the structural conditions articulated by Courbage et
al.~\cite{Courbage2005}. This approach facilitates a synthesis of
theoretical and computational techniques, promoting a deeper
convergence between classical ergodic theory and modern
entropy-based methodologies.
\subsubsection*{Relevance and Applications} The bounds and continuity results presented here contribute to a more nuanced understanding of entropy anisotropy in spatially extended systems, with direct implications for:
\begin{itemize}
    \item Modeling information transfer in biological and cellular networks,
    \item Quantifying uncertainty in spatially distributed dynamical processes.
\end{itemize}

\section{Topological Directional Entropy}\label{topological-directional-entropy}

This section investigates the \textbf{TDE} for
\(\mathbb{Z}^2\)-actions induced by 1D CAs and shift
transformation. The analysis extends the methodological framework
proposed by D’Amico et al. \cite{Damico2003}. Building on this
foundation, prior research \cite{Akin2009-DEnt} derived explicit
formulas for the TDE of \(\mathbb{Z}^2\)-actions governed by
linear LCAs and shifts over the ring
\(\mathbb{Z}_m\), establishing theoretical benchmarks for such
systems.

For structurally complex or nonlinear CA, Courbage
\cite{Courbage2002} introduced a generalized measure of
spatiotemporal complexity. This approach quantifies directional
entropy through "windows" aligned with the \(\theta\)-axis in the
space-time plane (Fig.~\ref{fig1Cregion}), capturing entropy
anisotropy and directional dependencies in system evolution.

Essentially the following key contributions will be presented:
\begin{itemize}
    \item \textbf{Extended Methodological Framework}: Generalizes D’Amico et al.’s approach \cite{Damico2003} to compute TDE for \(\mathbb{Z}^2\)-actions induced by 1D CA and shifts.
    \item \textbf{Explicit Entropy Formulas}: Derives closed-form TDE expressions for \(\mathbb{Z}^2\)-actions driven by linear CA over \(\mathbb{Z}_m\) \cite{Akin2009-DEnt}.
    \item \textbf{Anisotropy Quantification}: Proposes Courbage’s directional "window" method \cite{Courbage2002} to analyze entropy anisotropy in nonlinear CA via space-time directional windows (Fig.~\ref{fig1Cregion}).
    \item \textbf{Unified Benchmarks}: Establishes continuity and bounds for directional entropy, unifying linear and nonlinear CA analysis.
\end{itemize}
%
\begin{figure} [!htbp]
\centering
\includegraphics[width=55mm]{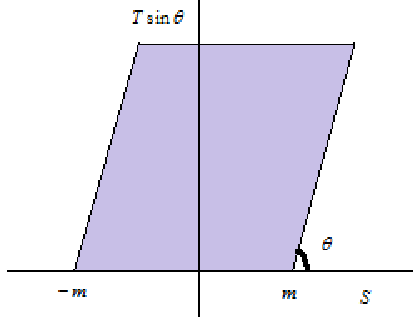}\ \
\caption{The parallelogram $W(m, T, \theta$)}\label{fig1Cregion}
\end{figure}

Consider an orbit of the 1D CA \( T_{f[l,r]} \) in \(
\mathbb{Z}_p^{\mathbb{Z}} \), which is represented by the sequence
\( \{x^{t}\}_{t \in \mathbb{Z}^{+}} = \{x^{t}_{n}\}_{n \in
\mathbb{Z}}, t \in \mathbb{Z}^{+} \), where \( x^{t+1}_{n} =
f(x^{t}_{n+l}, \dots, x^{t}_{n+r}) \). The window shown in Fig.
\ref{fig1Cregion} is defined by:

\[
W(m,T,\theta) = \{(x + y \cos \theta, y \sin \theta) : -m < x < m,
0 < y \leq T \}.
\]

\begin{definition} \cite{Courbage21}
Given \( W(m,T,\theta) \) and \( \epsilon > 0 \), a set \( M
\subset \mathbb{Z}_p^{\mathbb{Z}} \) is called\\ \(
\left(\epsilon, W(m,T,\theta)\right) \)-separated if for every
pair \( x, \overline{x} \in M \), there exists \( (n,t) \in
W(m,T,\theta) \) such that:

\[
d\left( \left( T^{t}_{f[l,r]} x \right)_n, \left( T^{t}_{f[l,r]}
\overline{x} \right)_n \right) \geq \epsilon,
\]

where

\[
d(x,y) = \sum_{i \in \mathbb{Z}} \frac{|x_i - y_i|}{p^{|i|}},
\quad \text{for} \quad x, y \in \mathbb{Z}_p^{\mathbb{Z}}.
\]

\end{definition}

The number of distinct orbits with accuracy \( \epsilon \) in \(
W(w,T,\theta) \) is defined as:

\[
N(\epsilon, W(w,T,\theta)) = \max \{ \#(M) : M \text{ is an }
(\epsilon, W(m,T,\theta)) \text{-separated set} \}.
\]

The limit:
\[
h_{\theta}(\mathbb{Z}_p^{\mathbb{Z}^2}, \Phi) = \lim_{\varepsilon
\to 0} \left( \overline{\lim_{m \to \infty}} \frac{1}{2m+1} \left(
\overline{\lim_{T \to \infty}} \frac{1}{T \sin \theta} \ln \left(
N(\epsilon, W(m,T,\theta)) \right) \right) \right)
\]
is referred to as the density of the TDE of \(
(\mathbb{Z}_p^{\mathbb{Z}^2}, \Phi) \) in the direction \( \theta
\) (see \cite{Sinai1985} for details).

Let us consider the local role
\begin{equation}\label{bipermutative}
f(x_{-l},\ldots,x_r)=\sum\limits_{i=-l}^r a_ix_i\mod p,
\end{equation}
where $a_{-l}\neq 0$, $a_{r}\neq 0$ and $\gcd(a_{-l},p)=1$,
$\gcd(a_{r},p)=1$.

The formulae of the DE for bipermutative CAs has
been given by Milnor \cite{Milnor1988}, for the CA associated with
the local rule \eqref{bipermutative}. In the special case $l < 0 <
r$, Bernardo and Courbage \cite{Bernardo-Co-2001} have provided
another expanded proof for 1D CA with the bipermutative local
rule.

Denoting $\cot \theta_{l} = -l, \cot \theta_{r}=-r$, Courbage
\cite{Courbage2002} has presented an explicit formula for the DE
by:
\begin{equation}\label{top-direc-ent1}
h_{\theta }(\mathbb{Z}_p^{\mathbb{Z}^2},T_{f[l,r]})=\left\{
\begin{array}{ll}
 \text{(cos(}\theta )+r\sin (\theta ))\ln p, & \text{for } \theta \in [0,\theta _l], \\
 (r-l)\sin (\theta )\ln p, & \text{for } \theta \in [\theta _l,\theta _r], \\
 \text{(cos(}\theta )+l\sin (\theta ))\ln p, & \text{for } \theta \in [\theta _r,\pi].
\end{array}
\right.
\end{equation}
\begin{remark}
If the angle $\theta =\pi/2$, the equation \eqref{top-direc-ent1}
coincides with the formula of the topological entropy of 1D LCA.
\end{remark}

\begin{example}\label{ex:TDE-p=5} Let us consider the local rule given by
\begin{equation}\label{biper-rule-l=-4-r=3}
f(x_{-4},x_{-3},\cdots ,x_2,x_3)=3x_{-4}+2x_{-3}+3x_2+4x_3\pmod 5.
\end{equation}
Then, it is clear that the CA $T_{f[-4,3]}$ is
bipermutative. 
From \eqref{top-direc-ent1}, we compute the TDE as:
\begin{align}\label{Ex:Top-Ent-1}
h_{\theta }(\mathbb{Z}_5^{\mathbb{Z}^2},\Phi )&= \left\{
\begin{array}{ll}
 (\cos \theta + 3\sin \theta )\log 5, & \text{if } 0\leq \theta \leq \cot^{-1} (4), \\
 7\sin \theta \log 5, & \text{if }  \cot^{-1} (4) \leq \theta \leq \pi-\cot^{-1} (3), \\
 |\cos \theta - 4\sin \theta |\log 5, & \text{if }  \pi-\cot^{-1} (3) \leq \theta \leq \pi.
\end{array}
\right.
\end{align}

This expression defines the DE as a piecewise function of the
angle $\theta$ over distinct angular intervals. Figure~\ref{direc-ent-segment1} illustrates the critical angular
boundaries at $\theta_{-4} = \arcsin\left(1/\sqrt{17}\right)$ and
$\theta_3 = \pi - \arcsin\left(1/\sqrt{10}\right)$. As established
in~\cite{Akin2009-DEnt}, a 1D LCA governed by a bipermutative
local rule divides the angular domain into three characteristic
regions. The directional entropy segments corresponding to the
combined transformation of the CA (defined by the local rule
in~\eqref{top-ent-ex2}) and shift map  are presented in
Figure~\ref{direc-ent-segment1a}. Furthermore,
Figure~\ref{plot-DirectionalEnt1} displays the complete entropy
profile $h(\theta)$ across the interval $\theta \in [0, \pi]$.

We emphasize that these angular sectors are completely
characterized by the underlying local rules of the CA.

\begin{figure} [!htbp]\label{direc-ent-segment1}
\centering
\includegraphics[width=75mm]{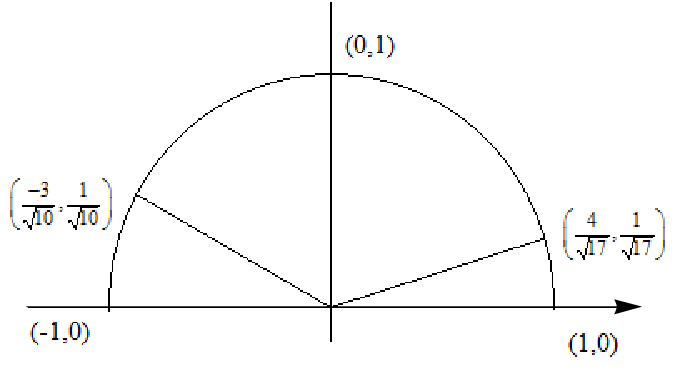}\ \
\caption{The sectors corresponding to angles $\theta
_{-4}=\arcsin(\frac{1}{\sqrt{17}})$ and $\theta
_3=\pi-\arcsin(\frac{1}{\sqrt{10}})$. }\label{direc-ent-segment1a}
\end{figure}

\begin{figure} [!htbp]
\centering
\includegraphics[width=75mm]{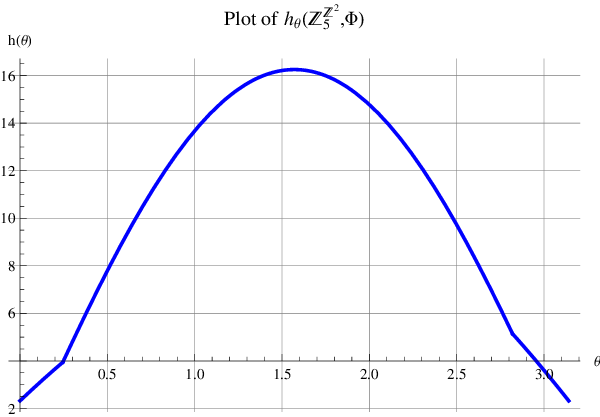}\ \
\caption{The graph of the function given in \eqref{Ex:Top-Ent-1}
for bipermutative CA generated by the local rule given by
\eqref{biper-rule-l=-4-r=3}.}\label{plot-DirectionalEnt1}
\end{figure}
\end{example}

\subsection{DTE for ring $m=p^k$} In the following Lemma we present a DTE formula for the ring $m=p^k$.
\begin{lemma}(\cite{Akin2009-DEnt})\label{Akin2009-DEnt-lemma2}\ Let $T_{f[-r,r]}$ be 1D LCA over the ring $\mathbb{Z}_{p^{k}}$ with local rule
\begin{equation}\label{local-rule-p-k}
f(x_{-r},\ldots,x_r)=\sum\limits_{i=-r}^r a_ix_i\mod p^k,
\end{equation}
where $p$ is prime number, $a_{-r}\neq 0$, $a_{r}\neq 0$. Assume
that $f$ is both rightmost and leftmost permutative. Define
\begin{equation}\label{Left-Right-perm1}
P=\{0\}\cup \{j:gcd(a_j, p)=1 \} =\{j_1,j_2,\cdots, j_t\},\ L=\min
P and\ R= \max P
\end{equation}
\begin{equation}
\textbf{P}=\{arccot(-L)=\theta_L,arccot(-R)=\theta_R \}
\end{equation}
Assume that $\Phi$ is a $\mathbb{Z}^2$-action generated by
$T_{f[-r,r]}$ and $\sigma$. Then we have
\begin{equation}\label{top-direc-ent2}
h_{\theta}(\mathbb{Z}^{\mathbb{Z}^2}_{p^{k}}, \Phi) =
\begin{cases}
k |(\cos(\theta) + R \sin(\theta))| \log p, & \text{for } \theta \in [0, \theta_L], \\
k |(R - L) \sin(\theta)| \log p, & \text{for } \theta \in [\theta_L, \theta_R], \\
k |(\cos(\theta) + L \sin(\theta))| \log p, & \text{for } \theta
\in [\theta_R, \pi].
\end{cases}
\end{equation}
\end{lemma}
\begin{proof}
It is well known that if the local rule $f$ has radius $r$, i.e.,
it depends on at most $2r+1$ variables. We associate to the local
rule $f$ the finite \emph{fps} $F(X)=\sum_{i=-r}^{ r}a_iX^{i}$.
Also, finite fps associated with $f^{(n)}$ is $F^{n}(X)$. It is
obvious that there exist $L,R\in\mathbb{Z}$ and $n\in\mathbb{N}$
such that $f^{(n)}$ is permutative in the variables $x_L$ and
$x_R$. Thus, the local rule $f^{(n)}$ does not depend on variables
$x_j$ with $j<L$ or $j>R$, in other words, we can ignore the
variables $x_j$ with $j<L$ or $j>R$. Then, the local rule
$f^{(n)}$ is both leftmost and rightmost permutative. From
\eqref{top-direc-ent1}, we have the desired formula
\eqref{top-direc-ent2}. Thus, the proof is completed.
\end{proof}

The function \( h_{\theta}(\mathbb{Z}^{\mathbb{Z}^2}_{p^{k}},
\Phi) \) is defined as a piecewise function that depends on the
angle \( \theta \). Now, let us consider the local rule
\begin{equation}\label{local-rule-p-k3}
f(x_{l},\ldots,x_r)=\sum\limits_{i=l}^r a_ix_i\pmod m,
\end{equation}
where $m=\prod_{i=1}^{h}p_{i}^{k_{i}}$, the $p_{i}$ is prime
number.
\begin{corollary}
Let us consider the local rule given in \eqref{local-rule-p-k}. If
$p$ divides all coefficients $a_i$ for all $i=l,\cdots,r$, then
$h_{\theta} (\mathbb{Z}^{\mathbb{Z}^2}_{p^{k}},\Phi)=0.$
\end{corollary}

\begin{example}\label{ex:TDE-m=9}
Consider the 1D LCA over the ring $\mathbb{Z}_9$ with the following local rule:
\begin{equation}\label{biper-rule-l=-5-r=6}
f(x_{-5},x_{-4},\ldots,x_5,x_6) = 3x_{-5} + 2x_{-4} + 5x_5 + 7x_6 \pmod{9}.
\end{equation}

We identify the positions $j$ for which $\gcd(a_j, 3) = 1$, since $p^k = 3^2 = 9$:
\[
P = \{ j : \gcd(a_j, 3) = 1 \} = \{-4, 5, 6\}.
\]
Therefore, we have:
\[
L = \min P = -4, \qquad R = \max P = 6.
\]

We now define the angular parameters based on the inverse cotangent:
\[
\theta_L = \cot^{-1}(-L) = \cot^{-1}(4), \qquad \theta_R = \cot^{-1}(-R) = \cot^{-1}(-6).
\]

According to the known result for directional entropy (see Lemma~\ref{Akin2009-DEnt-lemma2}), the TDE for the function $\Phi$ is given by:
\begin{equation}\label{entropy-function-9}
h_\theta(\mathbb{Z}^{\mathbb{Z}^2}_9, \Phi) =
\begin{cases}
2 \left| \cos(\theta) + 6 \sin(\theta) \right| \log 3, & \text{for } \theta \in [0, \cot^{-1}(4)]; \\[6pt]
20 \left| \sin(\theta) \right| \log 3, & \text{for } \theta \in [\cot^{-1}(4), \cot^{-1}(-6)]; \\[6pt]
2 \left| \cos(\theta) - 4 \sin(\theta) \right| \log 3, & \text{for } \theta \in [\cot^{-1}(-6), \pi]
\end{cases}
\end{equation}

This piecewise function reflects the structure of entropy in different directional ranges for the given rule over $\mathbb{Z}_9$.
\begin{figure} [!htbp]
\centering
\includegraphics[width=100mm]{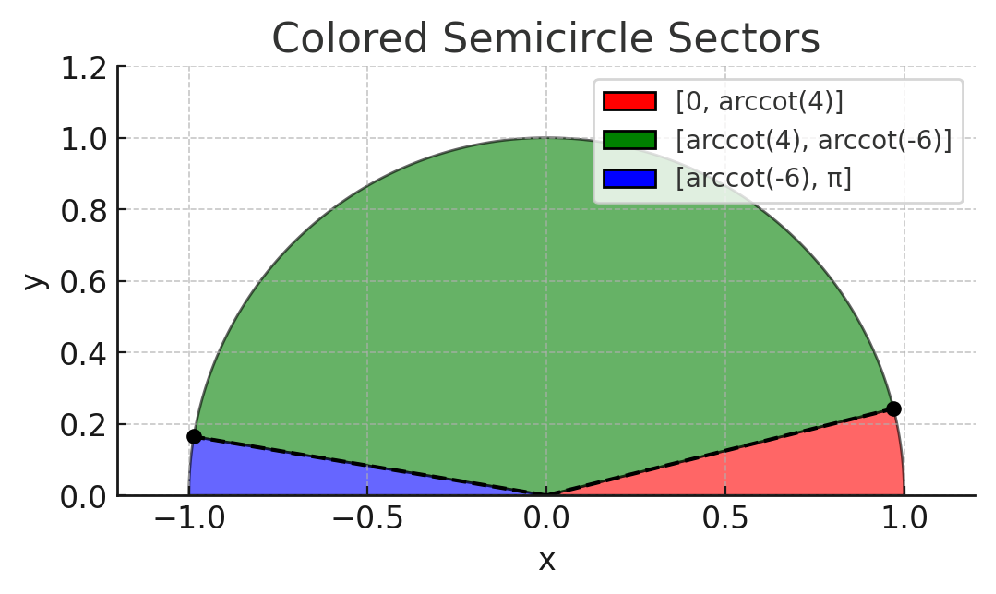}\ \
\caption{The figure shows the sectors for the angles expressing
the right and left permutative values corresponding to prime
number $p=3$  in local rule \eqref{biper-rule-l=-5-r=6}.
}\label{TDE-Ex-9-sector}
\end{figure}

\begin{figure} [!htbp]
\centering
\includegraphics[width=80mm]{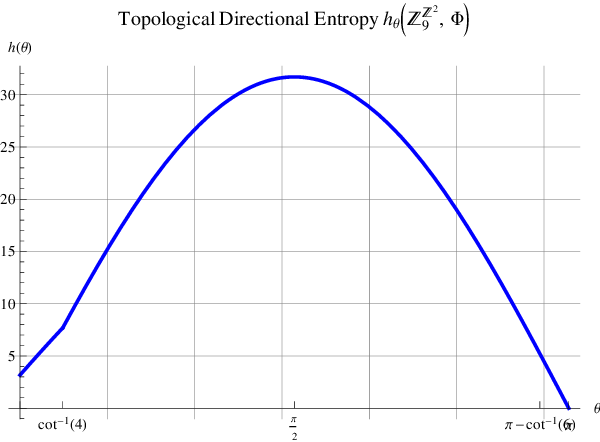}\ \
\caption{The graph of the TDE of $\mathbb{Z}^{2}$-action generated
by 1D LCA associated with the local rule given by
\eqref{biper-rule-l=-5-r=6} as a function of $\theta
$.}\label{Top-dir-entropy-9}
\end{figure}
\end{example}

\subsection{DTE for ring $m=p_1^{k_1}.p_2^{k_2} \ldots p_h^{k_h}$} 

The proof is provided similarly to Theorem \ref{Thm-MTDE1q} (1).

\begin{theorem}{\cite[Theorem 4.4]{Akin2009-DEnt}}\label{tde}\ Suppose that for $i = 1,\ldots,h$ the rule defined by \eqref{local-rule-p-k3} is both left and right permutative. Let $T_{f[l, r]}$ be a 1-D LCA over $\mathbb{Z}^{\mathbb{Z}}_{m}$ with local rule defined by (\ref{local-rule-p-k3}) and let $m=p_1^{k_1}.p_2^{k_2} \ldots p_h^{k_h}$ be the prime factor decomposition of $m$. For $i = 1,\ldots, h$ define
$$
P_{i} =\{0\}\cup \{j: gcd (\lambda_{j}, p_{i})=1\}, L_{i} = \min
P_{i}, R_{i} = \max P_{i}.
$$
$$
\textbf{P}_i=\{\theta_{L_i}=\text{arccot}(-L_i),\theta_{R_i}=\text{arccot}(-R_i)\}.
$$
Then we have
\begin{align*}\label{tde4}
&h_{\theta}(\mathbb{Z}^{\mathbb{Z}^2}_{m},\Phi)= \sum\limits_{i=1}^h h_{\theta }(\mathbb{Z}_{p_i^{k_i}}^{\mathbb{Z}^2},\Phi_{p_i}) \\
&= \left\{
\begin{array}{ll}
\sum\limits_{i=1}^h k_i |\cos \theta + R_i \sin \theta| \log p_i, & \text{for } \theta \in [0, \min \theta_{L_{j}}], \\
\sum\limits_{i=1, i \neq j}^h k_i |\cos \theta + R_i \sin \theta| \log p_i + \\
k_j |(R_j - L_j) \sin \theta| \log p_j, & \text{for } \theta \in [\min \theta_{L_{j}}, \min \{\theta_{R_{j_1}}, \theta_{L_{j_1}}\}], \\
\vdots \\
\sum\limits_{i=1}^h k_i |\cos \theta + L_i \sin \theta| \log p_i,
& \text{for } \theta \in [\max \theta_{R_{j}}, \pi].
\end{array}
\right.
\end{align*}
\end{theorem}
\begin{proof}
For all $i=1,\cdots,h$ define the local rules
$f_{p_i}:\mathbb{Z}_{p_i^{k_i}}^{r-l+1}\rightarrow
\mathbb{Z}_{p_i^{k_i}}$ by
\begin{eqnarray*}
f_{p_i}(x_l,\cdots ,x_r)&=&\sum _{j=l}^r c^{(i)}_jx_j\pmod {p_i^{k_i}}\\
&\equiv& f(x_l,\cdots ,x_r)\pmod {p_i^{k_i}}.
\end{eqnarray*}
Then for all $i=1,\cdots,h$, $f_{p_i}$ generates 1D LCA $T_{f_{p_i}[l,r]}:\mathbb{Z}_{p_i^{k_i}}^{\mathbb{Z}}\rightarrow \mathbb{Z}_{p_i^{k_i}}^{\mathbb{Z}}.$\\
We observe that $\mathbb{Z}_m^{\mathbb{Z}^2}\cong
\mathbb{Z}_{p_1^{k_1}}^{\mathbb{Z}^2}\times
\mathbb{Z}_{p_2^{k_2}}^{\mathbb{Z}^2}\times \cdots \times
\mathbb{Z}_{p_h^{k_h}}^{\mathbb{Z}^2}$ induces an isomorphism
$$\Psi :\mathbb{Z}_m^{\mathbb{Z}^2}\rightarrow
\mathbb{Z}_{p_1^{k_1}}^{\mathbb{Z}^2}\times
\mathbb{Z}_{p_2^{k_2}}^{\mathbb{Z}^2}\times \cdots \times
\mathbb{Z}_{p_h^{k_h}}^{\mathbb{Z}^2}.
$$
For all $i=1,\cdots,h$ define $\mathbb{Z}^2$-actions $\Phi
_{p_i}:\mathbb{Z}_{p_i^{k_i}}^{\mathbb{Z}^2}\rightarrow
\mathbb{Z}_{p_i^{k_i}}^{\mathbb{Z}^2}$ by $\Phi
_{p_i}^{(u,s)}=T_{f_{p_i}[l,r]}^s\circ \sigma ^u$ for
$(u,s)\in\mathbb{Z}^2$. So, we have the following commutative
diagram
\begin{equation}\label{diagram-TDE1}
\vcenter{\xymatrix@C=3.5em@R=3em{
\mathbb{Z}_m^{\mathbb{Z}^2} \ar[r]^{\Phi} \ar[d]_{\Psi} & \mathbb{Z}_m^{\mathbb{Z}^2} \ar[d]^{\Psi} \\
\mathbb{Z}_{p_1^{k_1}}^{\mathbb{Z}^2}\times
\mathbb{Z}_{p_2^{k_2}}^{\mathbb{Z}^2}\times \cdots \times
\mathbb{Z}_{p_h^{k_h}}^{\mathbb{Z}^2}
\ar[r]_-{\bigtimes_{i=1}^h \Phi_{p_i}} & \mathbb{Z}_{p_1^{k_1}}^{\mathbb{Z}^2}\times
\mathbb{Z}_{p_2^{k_2}}^{\mathbb{Z}^2}\times \cdots \times
\mathbb{Z}_{p_h^{k_h}}^{\mathbb{Z}^2}}}
\end{equation}
From the diagram \eqref{diagram-TDE1}, we can prove that $\Psi
=\Psi _{p_1}\times \Psi _{p_2}\times \cdots \times \Psi _{p_h}$ is
an isomorphism. Therefore, $\Phi =\Phi _{p_1}\times \Phi
_{p_2}\times \cdots \times \Phi _{p_h}$. From the weak Addition
Theorem, we get
$$
h_{\theta}(\mathbb{Z}^{\mathbb{Z}^2}_{m},\Phi)=\sum _{i=1}^h
h_{\theta }(\mathbb{Z}_{p_i^{k_i}}^{\mathbb{Z}^2},\Phi _{p_i}).
$$
Thus the proof is completed.
\end{proof}

The theorem \ref{tde} gives an explicit formula for the TDE of the
$\mathbb{Z}^2$-actions determined by 1D LCAs and shift map over the ring $\mathbb{Z}_m$.

\begin{example}\label{ex:DE-30}
Let us consider the local rule
\begin{equation}\label{top-ent-ex2}
f(x_{-3},x_{-2},\cdots,x_2,x_3)=2x_{-3}+3x_{-2}+
5x_{-1}+30x_0+3x_1+2x_2+5x_3\pmod {30} \end{equation} From theorem
\ref{tde}, we have
\begin{eqnarray*}
P_{2} &=&
\{-2,-1,0,1,3\}, L_{2} = -2, R_{2} = 3, \textbf{P}_2=
\{\theta_{L_2}=\cot^{-1} (2),\theta_{R_2}=\cot^{-1} (-3)\}\\
P_{3} &=&
\{-3,-1,0,2,3\}, L_{3} = -3, R_{3} =3, \textbf{P}_3=
\{\theta_{L_3}=\cot^{-1} (3),\theta_{R_3}=\cot^{-1} (-3)\}\\
P_{5} &=&
\{-3,-2,0,1,2\}, L_{2} = -3, R_{5} = 2, \textbf{P}_5=
\{\theta_{L_5}=\cot^{-1} (3),\theta_{R_5}=\cot^{-1} (-2)\}\\
\end{eqnarray*}

Let us compute the MTDE of $Z^2$-action generated by the shift map
and 1D LCA associated with the local rule given by
\eqref{top-ent-ex2} as a piecewise function depending on the value
of \( \theta \). 

So, we get
\begin{align*}
h_{\theta}(\mathbb{Z}^{\mathbb{Z}^2}_{30}, \Phi)& =h_{\theta}(\mathbb{Z}^{\mathbb{Z}^2}_{2}, \Phi_1)+h_{\theta}(\mathbb{Z}^{\mathbb{Z}^2}_{3}, \Phi_2)+h_{\theta}(\mathbb{Z}^{\mathbb{Z}^2}_{5}, \Phi_3)\\
&=
\begin{cases} 
\left|\cos\theta + 3\sin\theta\right|\log 6 + \left|\cos\theta + 2\sin\theta\right|\log 5, & \theta \in [0, 0.3218], \\
\left|\cos\theta + 3\sin\theta\right|\log 2 + \sin\theta(6\log 3 + 5\log 5), & \theta \in [0.3218, 0.4636], \\
\sin\theta(5\log 2 + 6\log 3 + 5\log 5), & \theta \in [0.46, \pi - 0.46], \\
\sin\theta(5\log 2 + 6\log 3)+ \left|\cos\theta - 3\sin\theta\right|\log 5, & \theta \in [\pi - 0.46, \pi - 0.32], \\
\left|\cos\theta - 2\sin\theta\right|\log 2 + \left|\cos\theta - 3\sin\theta\right|\log 15, & \theta \in [\pi - 0.32, \pi].
\end{cases}
\end{align*}
\begin{figure} [!htbp]
\centering
\includegraphics[width=100mm]{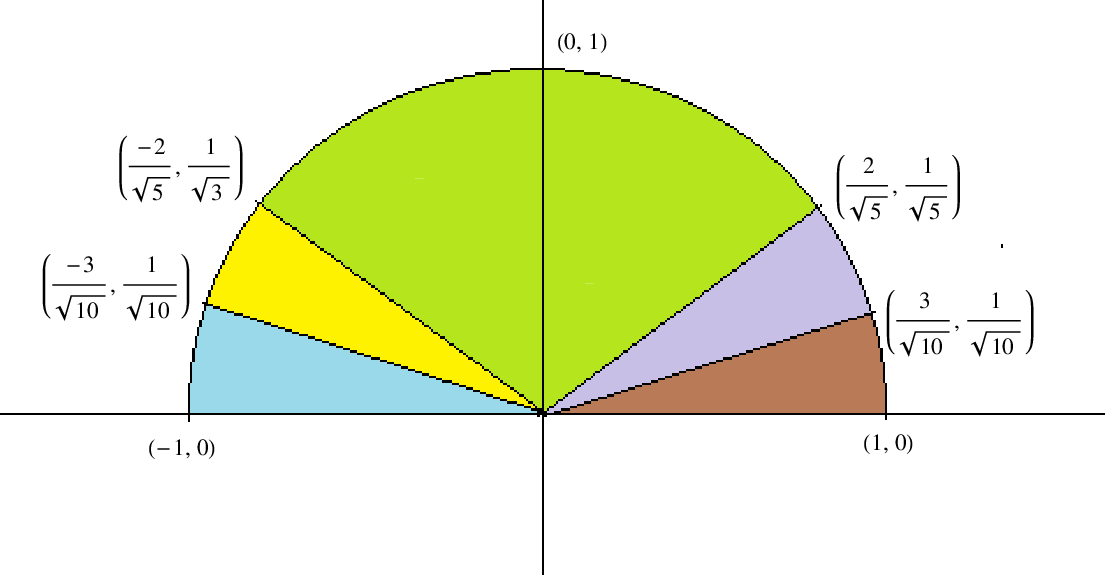}\ \
\caption{The figure shows the sectors for the angles expressing
the right and left permutative values corresponding to prime
numbers $p=2,3,5$  in local rule \eqref{top-ent-ex2}.
}\label{Top-dir-entropy30-sector}
\end{figure}

\begin{figure} [!htbp]
\centering
\includegraphics[width=80mm]{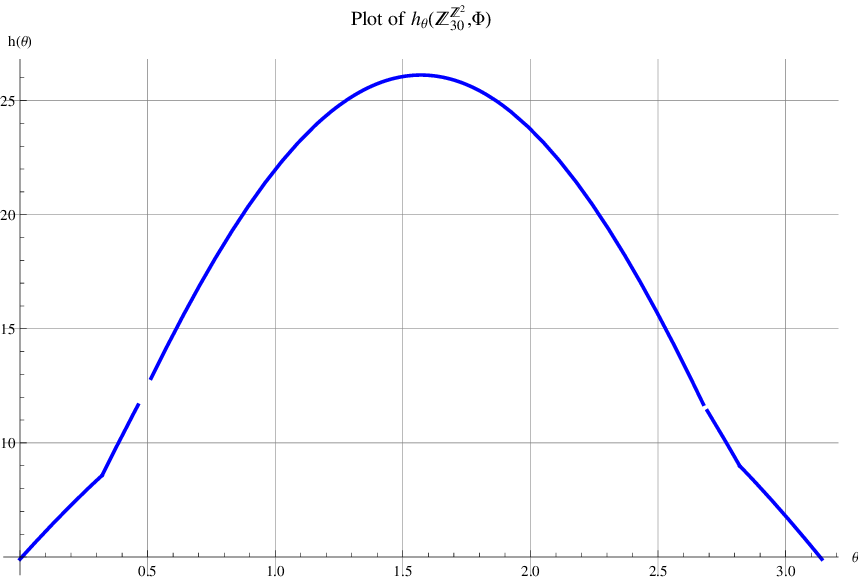}\ \
\caption{The graph of the TDE of $\mathbb{Z}^{2}$-action generated
by 1D LCA associated with the local rule given by
\eqref{top-ent-ex2} as a function of $\theta
$.}\label{Top-dir-entropy30}
\end{figure}

In figure \ref{Top-dir-entropy30-sector}, we have plotted the
sectors for the angles expressing the right and left permutative
values corresponding to prime numbers $p=2,3,5$  in local rule
\eqref{top-ent-ex2}. As shown in Figure \ref{Top-dir-entropy30},
this quantity of entropy
$h_{\theta}(\mathbb{Z}^{\mathbb{Z}^2}_{30},\Phi)$ takes the
largest value for $\theta =\pi/2$. The function
$h_{\theta}(\mathbb{Z}^{\mathbb{Z}^2}_{30},\Phi)$ reveals as a
piecewise function defined over 5 disjoint intervals. Another
interesting case is that this entropy function
$h_{\theta}(\mathbb{Z}^{\mathbb{Z}^2}_{30},\Phi)$ is discontinuous
at $\theta =\pi -0.46$ and $\theta =\pi -0.32$ values.
\end{example}

\paragraph{Key Observations}  
A notable pattern emerges: across both case studies, the TDE achieves its supremum at the orthogonal orientation $\theta = \pi/2$, where it aligns precisely with the system's topological entropy.  

\paragraph{Computational Scope}  
In this section, we have evaluated the TDE for $\mathbb{Z}^2$-actions generated by 1D-CAs driven by three distinct local update rules, coupled with the shift transformation $\sigma$. The configurations span:  
\begin{itemize}  
    \item The prime field $\mathbb{F}_5$ (Example~\ref{ex:TDE-p=5}),  
    \item The modular ring $\mathbb{Z}_9$ (Example~\ref{ex:TDE-m=9}),  
    \item The composite ring $\mathbb{Z}_{30}$ (Example~\ref{ex:DE-30}).  
\end{itemize}

\subsection{Comments and references}

The results presented in Section
\ref{topological-directional-entropy} are based on previous studies
\cite{Akin2009-DEnt,Bernardo-Co-2001,Courbage2002,Milnor1988}. It
is generally impossible to compute the TDE of
\(\mathbb{Z}^{2}\)-actions through algorithmic methods
\cite{Courbage2005}. This section explores the concept of TDE as
discussed in \cite{Akin2009-DEnt,Akin2008,Bernardo-Co-2001}.
Future research will focus on addressing unresolved questions
regarding the computation of TDE for \(\mathbb{Z}^{2}\)-actions
defined over various rings.

In \cite{Akin-95-Erg-AMC}, we investigated the ergodic properties
of \(\mathbb{Z}^2\)-actions generated by shifts and CAs governed
by \( r \)-radius local maps on the space \(
\mathbb{Z}_m^{\mathbb{Z}} \). This study lays the groundwork for
further exploration of DEs associated with this family of transformations. In Ref. \cite{AS2007}, the authors examined the topological entropy of cellular automata defined over the class of Galois rings, in contrast to the standard ring $\mathbb{Z}_m$. Using similar reasoning, one can formulate explicit expressions for the computation of TDE.

Additionally, in \cite{ADGT-2023}, we determined the algebraic
entropy of finitary 1D LCAs \( S \). Our findings demonstrated
that the algebraic entropy of \( S \) aligns with the topological
entropy of its Pontryagin dual, \( T = \hat{S} \), in accordance
with the Bridge Theorem. However, the computation of directional
algebraic entropy for the \(\mathbb{Z}^2\)-action induced by
finitary LCAs remains an open problem, which we plan to
investigate in future studies.  When the direction vector \( v =
(0,1) \) is chosen, the notion of directional entropy reduces to
the standard (classical) entropy of the transformation.
Consequently, the findings presented in this section can be viewed
as generalizations of the results established in \cite{Akin-2024}.

 \bibliographystyle{plain}

\end{document}